\documentclass{amsart}

\usepackage{amsmath}
\usepackage{amssymb}
\usepackage{commath}
\usepackage{mathtools}
\usepackage[all]{xy}
\usepackage{graphicx} 

\usepackage{bigints}

\usepackage{hyperref}

\theoremstyle{plain}
\newtheorem{theorem}{Theorem}
\newtheorem{lemma}[theorem]{Lemma}
\newtheorem{proposition}[theorem]{Proposition}
\newtheorem{corollary}[theorem]{Corollary}
\theoremstyle{definition}
\newtheorem{definition}[theorem]{Definition}
\theoremstyle{remark}
\newtheorem{remark}[theorem]{Remark}

\numberwithin{theorem}{subsection}
\numberwithin{equation}{section}

\newcommand{\B}{\mathbb{B}}
\newcommand{\D}{\mathbb{D}}

\newcommand{\HH}{\mathbb{H}}
\newcommand{\Sbb}{\mathbb{S}}
\newcommand{\Mbb}{\mathbb{M}}
\newcommand{\R}{\mathbb{R}}

\newcommand{\Z}{\mathbb{Z}}

\newcommand{\Spe}{\mathrm{Sp}}

\newcommand{\spe}{\mathfrak{sp}}

\newcommand{\re}{\operatorname{Re}}

\newcommand{\cF}{\mathcal{F}}
\newcommand{\cM}{\mathcal{M}}

\newcommand{\downmapsto}{\rotatebox[origin=c]{-90}{$\scriptstyle\mapsto$}\mkern2mu}

\begin{document}

\title[Geometry of regular M\"obius and the unit ball]{Geometry of the slice regular M\"obius transformations of the quaternionic unit ball}

\author{Raul Quiroga-Barranco}
\address{Centro de Investigaci\'on en Matem\'aticas \\ Guanajuato \\ M\'exico}
\email{quiroga@cimat.mx}

\begin{abstract}
	For the quaternionic unit ball $\mathbb{B}$, let us denote by $\mathcal{M}(\mathbb{B})$ the set of slice regular M\"obius transformations mapping $\mathbb{B}$ onto itself. We introduce a smooth manifold structure on $\mathcal{M}(\mathbb{B})$, for which the evaluation(-action) map of $\mathcal{M}(\mathbb{B})$ on $\mathbb{B}$ is smooth. The manifold structure considered on $\mathcal{M}(\mathbb{B})$ is obtained by realizing this set as a quotient of the Lie group $\mathrm{Sp}(1,1)$, Furthermore, it turns out that $\mathbb{B}$ is a quotient as well of both $\mathcal{M}(\mathbb{B})$ and $\mathrm{Sp}(1,1)$. These quotients are in the sense of principal fiber bundles. The manifold $\mathcal{M}(\mathbb{B})$ is diffeomorphic to $\mathbb{R}^4 \times S^3$.
\end{abstract}



\maketitle

\section{Introduction}
\label{sec:intro}
The hyperholomorphic function theory of a quaternionic variable is a very interesting and fruitful topic of analysis. The notion of slice regularity, also known as Cullen regularity, is particularly useful as a tool to understand non-commutative function theory and functional analysis. The reason is that slice regularity nicely adapts to quaternions (and other algebras) while maintaining a close enough relation with the classical complex analysis. This has lead to a hyperholomorphic theory that naturally, but non-trivially, generalizes many of the properties and results that hold for complex holomorphic functions. It goes without saying that these generalizations are indeed non-trivial, since the non-commutativity of quaternions poses problems that sometimes cannot be solved with classical techniques. However, this contributes to the richness of hyperholomorphic theory and the discovery of interesting new structures.

For the unit disk $\D$ in the complex plane, the notion of M\"obius transformation yields an important object that lies at the intersection of geometry and analysis. For the case given by $\D$, the holomorphicity of M\"obius transformations is elementary but fundamental to understand the structure of $\D$ and its associated mathematical constructions. In the unit ball $\B$ of the quaternions $\HH$ there is a straightforward generalization of M\"obius transformation that we call in this work classical quaternionic. We refer to section~\ref{sec:classicalMobius} for further details, where we denote by $\Mbb(\B)$ the set of all classical (quaternionic) M\"obius transformations that preserve $\B$. From the point of view of geometry, $\Mbb(\B)$ is important since it leads to a pair of fundamental structures: $4$-dimensional real hyperbolic and $1$-dimensional quaternionic hyperbolic geometries (see~Remark~\ref{rmk:BasSp(1,1)quotient}). This is possible because $\Mbb(\B)$ carries a Lie group structure that comes from the (matrix) Lie group $\Spe(1,1)$. In fact, as it is reviewed in section~\ref{sec:classicalMobius}, the Lie group $\Mbb(\B)$ is a quotient of $\Spe(1,1)$ (see~subsection~\ref{subsec:classicalMobius}) and $\B$ is a quotient of both $\Mbb(\B)$ and $\Spe(1,1)$ by suitable closed subgroups (see~subsection~\ref{subsec:Bquotientclassical}).

However, the Lie group $\Mbb(\B)$ of classical (quaternionic) M\"obius transformations does not entitle any sort of hyperholomorphic analysis. This is related to the fact that $4$-dimensional real hyperbolic and $1$-dimensional quaternionic geometries have no apparent connection with hyperholomorphic function theory. For further details on the last claim, we refer to \cite[Remark~6.11]{QB-SliceKahler}.

Using the $*$-regular product, it is possible to construct a slice regular generalization of the classical quaternionic M\"obius transformations. These so-called (slice) regular M\"obius transformations were introduced in~\cite{StoppatoMobius}. They have been the object of a profound study. We refer to \cite{GentiliStoppatoStruppa2ndEd}, and the bibliography therein, where most of the known results about them is discussed. In subsection~\ref{subsec:sliceregularity} we denote by $\cM(\B)$ the set of all (slice) regular M\"obius transformations that preserve $\B$ and recall some of their properties. Right from the definition of the elements in $\cM(\B)$ one can observe that this transformations are defined using the Lie group $\Spe(1,1)$.

From the previous discussion it is natural to consider the problem of introducing a Lie group structure on $\cM(\B)$. The reason is that for $\Mbb(\B)$ the existence of such structure leads to very interesting conclusions. However, it is well known that $\cM(\B)$ is not a group for the usual composition of maps. Also, other known notions of regular composition or product do not lead to a group structure (see~Remark~\ref{rmk:cM(B)notgroup}). Furthermore, we prove in Corollary~\ref{cor:cM(B)notgroup} that $\cM(\B)$ does not admit any group structure for which the natural map that defines it from $\Spe(1,1)$ is a homomorphism (or an anti-homomorphism). At this stage, it seems that $\cM(\B)$ cannot be endowed with a useful group structure.

Despite this obstruction, the goal of this work is to prove that $\cM(\B)$ admits a smooth manifold structure whose properties generalize those observed above and discussed with detail below for $\Mbb(\B)$. We achieved this so that $\cM(\B)$ and $\B$ can be related by differential geometric constructions while maintaining the spirit of hyperholomorphic function theory.

We introduce in Theorem~\ref{thm:cM(B)manifold} the manifold structure on $\cM(\B)$ studied in this work. We achieve this by identifying $\cM(\B)$ with a quotient of the Lie group $\Spe(1,1)$ by the closed subgroup $\Spe(1) I_2$. This leads right away to a Lie theoretical  explanation of why $\cM(\B)$ is not a Lie group: $\Spe(1)$ is not a central subgroup of $\HH\setminus\{0\}$ because the quaternions are not commutative. We refer to Remark~\ref{rmk:cM(B)manifold} for further details and a comparison with the classical case of $\Mbb(\B)$. On the other hand, Theorem~\ref{thm:cM(B)isBxSp(1)asmanifolds} shows that the natural map that characterizes the elements of $\cM(\B)$ as expressions of the form
\[
	(1 - q \overline{a})^{-*}*(q - a)u
\]
is in fact a diffeomorphism, thus showing that $\cM(\B)$ is diffeomorphic to the manifold $\R^4 \times S^3$. It also proves that our manifold structure on $\cM(\B)$ is closely related to the hyperholomorphic analysis of slice regular functions.

For the classical case of $\Mbb(\B)$, the natural map $\Spe(1,1) \rightarrow \Mbb(\B)$ that assigns to a matrix the corresponding classical (quaternionic) M\"obius transformation is in fact a (smooth) homomorphism of Lie groups (see~Proposition~\ref{prop:AtoFAhomomorphism}). From the previous discussion, it is clear that the natural map $\Spe(1,1) \rightarrow \cM(\B)$ that now assigns to a matrix the corresponding (slice) regular M\"obius transformation cannot be a homomorphism of Lie groups. However, we prove that this map is smooth. Furthermore, Corollary~\ref{cor:Sp(1,1)cM(B)PrincipalFiberBundle} proves that the map $\Spe(1,1) \rightarrow \cM(\B)$ is a so-called principal fiber bundle with structure group $\Spe(1)$. Subsection~\ref{subsec:principalbundles} reviews the required definitions and results needed to state and prove this claim and others below.

Next, we study the manifold $\cM(\B)$ as it relates to the unit ball $\B$. In the classical case, the Lie group $\Mbb(\B)$ has a smooth right action on $\B$ (see~subsection~\ref{subsec:Bquotientclassical}). Again, strictly speaking there cannot be a smooth action of $\cM(\B)$ on $\B$, because the former is not a Lie group. However, Theorem~\ref{thm:cM(B)action} shows that the map obtained by evaluating the elements of $\cM(\B)$ on $\B$ is smooth as a function defined on $\cM(\B) \times \B$. Hence, we do have a sort of generalized smooth action for the manifold structure introduced in $\cM(\B)$. Furthermore, we prove that the subset $\cM(\B)_0$ of transformations that fix the origin is both a compact submanifold of $\cM(\B)$ and a Lie group isomorphic to $\Spe(1)$. The use of principal fiber bundle constructions of the sort mentioned above, allows us to prove in Theorem~\ref{thm:BascM(B)quotient} that $\B$ is a smooth quotient of the manifold $\cM(\B)$ by an action of the Lie group $\cM(\B)_0$. Also, Corollary~\ref{cor:cM(B)asbundleoverB} shows that such quotient comes from a principal fiber bundle.

On the other hand, Theorem~\ref{thm:BasSp(1,1)quotient} provides a realization of $\B$ as a quotient of $\Spe(1,1)$ that comes from using $\cM(\B)$. It is remarkable that this yields a double coset quotient of $\Spe(1,1)$, i.e.~a quotient where one mods out by subgroups both on the left and right. Remark~\ref{rmk:BasSp(1,1)quotient-sliceregular} compares this phenomenon with the classical case of $\Mbb(\B)$ while explaining the reason for this behavior, which again boils down to the non-commutativity of $\HH$.

We would like to comment on the choice of maps in this work. More precisely, orbits of a transformation are computed using the inverse. Also, our constructions consider the maps $A \mapsto F_{A^{-1}}$ and $A \mapsto \cF_{A^{-1}}$ that assign to a matrix $A \in \Spe(1,1)$ the corresponding classical and slice regular, respectively, M\"obius transformation. Again, there is an inverse involved. We explain in Remark~\ref{rmk:FvsF-1} the reason for these choices, that are needed because of the non-commutativity of $\HH$ and the fact that $\cM(\B)$ is not a group.

It is worthwhile to mention a topic not considered here: Clifford M\"obius transformations. Their differential geometry, as it relates to certain Lie groups, is studied in \cite{Nolder}. In view of the notion of slice monogenic functions (see~\cite{ColomboSabadiniStruppaFunctionalBook}) and the developments presented here, it is natural to ask about the differential geometry of spaces of slice monogenic functions. We believe that this could be formulated into an interesting problem to solve elsewhere.

The author would like to thank Paula Cerejeiras and Fabrizio Colombo for bringing to his attention some facts and references that helped to improve this work.

\section{Classical M\"obius transformations}
\label{sec:classicalMobius}
We will denote by $\B$ the unit ball in the division algebra $\HH$ of quaternions. Let us also denote by $\Spe(1,1)$ the symplectic pseudo-unitary group with signature $(1,-1)$, which is defined by
\[
	\Spe(1,1) = \{ A \in M_2(\HH) : A^* I_{1,1} A = I_{1,1} \},
\]
where we take
\[
	I_{1,1} = 
	\begin{pmatrix*}[r]
		1 & 0 \\
		0 & -1
	\end{pmatrix*}.
\]
We will also consider the group of unitary quaternions defined by
\[
	\Spe(1) = \{ q \in \HH : |q| = 1\}.
\]
Both $\Spe(1,1)$ and $\Spe(1)$ are connected simple Lie groups, and the latter is compact. The former allows to define the \textbf{classical (quaternionic) M\"obius transformations}. More precisely, for a matrix
\[
	A = 
	\begin{pmatrix}
		a & c \\
		b & d
	\end{pmatrix} \in \Spe(1,1),
\]
the classical M\"obius transformation $F_A : \B \rightarrow \B$ associated to $A$ is given by
\begin{equation}\label{eq:FA}
	F_A(q) = (qc + d)^{-1} (qa + b).
\end{equation}
The collection of all classical M\"obius transformations will be denoted here by $\Mbb(\B)$. In other words, we have by definition
\[
	\Mbb(\B) = \{ F_A : A \in \Spe(1,1) \}.
\]
These transformations have been thoroughly studied in \cite{BisiGentiliMobius,BisiStoppatoMobius,CaoParkerWang,HarveySpinors,Helgason}, to name a few references. Here, we will present some of their basic well-known properties. Nevertheless, we will do so in a manner that will be useful later on. Hence, we will state known properties without proof, and refer to the standard works (e.g.~those mentioned above), while providing the arguments needed to complete less standard claims.

\subsection{Classical M\"obius transformations as a Lie group}
\label{subsec:classicalMobius}
In the first place, it is easy to see that $F_{I_2}$ is the identity map and also that
\begin{equation}\label{eq:FAB}
	F_{AB} = F_B \circ F_A,
\end{equation}
for every $A, B \in \Spe(1,1)$. In particular, $\Mbb(\B)$ is a group (for the product given by the composition of maps) and the assignment $A \mapsto F_A$ is an anti-homomorphism. It also follows that
\begin{equation}\label{eq:FA(-1)}
	(F_A)^{-1} = F_{A^{-1}},
\end{equation}
for every $A \in \Spe(1,1)$. We conclude as well that the assignment $A \mapsto F_A$ yields a right $\Spe(1,1)$-action on $\B$. Such action is well known to be smooth (i.e.~a Lie group action) and transitive.

We will consider the map
\begin{align}\label{eq:AtoFAhomomorphism}
	\Spe(1,1) &\rightarrow \Mbb(\B)  \\
		A &\mapsto F_{A^{-1}}, \notag
\end{align}
which is thus a homomorphism of groups. As an immediate consequence we obtain the next result.

\begin{proposition}\label{prop:AtoFAhomomorphism}
	The homomorphism of groups $\Spe(1,1) \rightarrow \Mbb(\B)$ given by \eqref{eq:AtoFAhomomorphism} has kernel $\Z_2 I_2$. This induces a natural isomorphism of groups
	\begin{align*}
		\Spe(1,1)/\Z_2 I_2 &\rightarrow \Mbb(\B)   \\
			A \Z_2 I_2 &\mapsto F_{A^{-1}}.
	\end{align*}
	In particular, $\Mbb(\B)$ has a Lie group structure for which \eqref{eq:AtoFAhomomorphism} is a homomorphism of Lie groups and $\Mbb(\B)$ has dimension $10$ as a manifold.
\end{proposition}
\begin{proof}
	Computation of the kernel of the first stated homomorphism, and thus the isomorphism, are simple exercises. The conclusion on the Lie group structure of $\Mbb(\B)$ follows from the fact that $\Z_2 I_2$ is a finite, and thus closed, normal subgroup of $\Spe(1,1)$. The last claim is a consequence of the fact that $\Spe(1,1)$ has dimension~$10$.
\end{proof}

\begin{remark}\label{rmk:AtoFAhomomorphism}
	Other works have already noted that the map $A \mapsto F_A$ is an anti-homomorphism because of the identity \eqref{eq:FAB}, which actually holds for arbitrary classical (quaternionic) M\"obius transformations of $\HH$. From this one can conclude (see for example~\cite{BisiGentiliMobius}) that the map $A \mapsto F_A$ allows to construct an anti-isomorphism between $\Spe(1,1)/\Z_2 I_2$ and $\Mbb(\B)$. However, we have proved in Proposition~\ref{prop:AtoFAhomomorphism} that they are isomorphic. These two seemingly opposite claims can be conciliated as follows. If $\varphi : G \rightarrow H$ is an anti-homomorphism of groups, then $\varphi(g^{-1}) = \varphi(g)^{-1}$, for every $g \in G$. Since $g \mapsto g^{-1}$ is an anti-isomorphism of $G$, we conclude that the map $g \mapsto \varphi(g^{-1}) = \varphi(g)^{-1}$ is now a homomorphism from $G$ to $H$. In other words, by taking an inverse on either group, two groups are isomorphic if and only if they are anti-isomorphic. This is precisely our construction.
\end{remark}

\subsection{The unit ball $\B$ as a quotient: classical case}
\label{subsec:Bquotientclassical}
There is a natural (left) action of $\Mbb(\B)$ on $\B$ given by the evaluation
\begin{align}\label{eq:M(B)action}
	\Mbb(\B) \times \B &\rightarrow \B  \\
		(F,q) &\mapsto F(q), \notag
\end{align}
which satisfies the properties stated in the next result. From now on, $\Spe(1) \times \Spe(1)$ will be considered as a diagonally embedded subgroup of $\Spe(1,1)$.

\begin{proposition}\label{prop:M(B)action}
	For the Lie group structure on $\Mbb(\B)$ given by Proposition~\ref{prop:AtoFAhomomorphism}, the $\Mbb(\B)$-action on $\B$ defined by \eqref{eq:M(B)action} is smooth and transitive. The subgroup $\Mbb(\B)_0$ of elements in $\Mbb(\B)$ that fix the origin is given by
	\[
		\Mbb(\B)_0 = \{ F_A : A \in \Spe(1) \times \Spe(1) \}.
	\]
	In other words, \eqref{eq:M(B)action} defines a transitive Lie group action of $\Mbb(\B)$ on $\B$ with stabilizer at $0$ given by $\Mbb(\B)_0$ as above.
\end{proposition}
\begin{proof}
	As noted before, the classical M\"obius transformations define a right $\Spe(1,1)$-action. It then follows from \eqref{eq:FA} that this $\Spe(1,1)$-action is smooth. It is also well known that the $\Spe(1,1)$-action is transitive. 
	
	We refer to \cite[Proposition~1.94]{KnappBeyond} for further details on the next claims. Since $\Z_2 I_2$ is discrete, the map considered in Proposition~\ref{prop:AtoFAhomomorphism} yields a covering map $\Spe(1,1) \rightarrow \Mbb(\B)$. By locally inverting the latter, we see that the smoothness of the $\Spe(1,1)$-action implies the smoothness of the $\Mbb(\B)$-action. Also, the transitivity of the latter is clearly an immediate consequence of the transitivity of the former.
	
	Finally, a straightforward computation using \eqref{eq:FA} shows that $F_A \in \Mbb(\B)_0$ if and only if $A$ is diagonal with diagonal entries that necessarily belong to $\Spe(1)$.
\end{proof}

Given Proposition~\ref{prop:M(B)action}, it is immediate to diffeomorphically identify $\B$ with a quotient of $\Mbb(\B)$ by the stabilizer subgroup $\Mbb(\B)_0$. This can be achieved by considering the orbit map for the origin $0$. However, we will consider a similar, but non-standard, orbit map to obtain such quotient. This approach will be better suited later on as we will see.

\begin{proposition}\label{prop:BasM(B)quotient}
	For the $\Mbb(\B)$-action on $\B$, the \textbf{inverse orbit map} $\Mbb(\B) \rightarrow \B$ at $0$  defined by $F \mapsto F^{-1}(0)$ is smooth and induces a well defined map
	\begin{align*}
		\Mbb(\B)_0 \backslash \Mbb(\B) &\rightarrow \B  \\
			\Mbb(\B)_0 F &\mapsto F^{-1}(0),
	\end{align*}
	which is a diffeomorphism.
\end{proposition}
\begin{proof}
	The result follows from Theorem~3.2 and Proposition~4.3 in \cite[Chapter~II]{Helgason} after a slight modification of their proofs as we now explain.
	
	With our current notation, it is easy to verify that the proof of \cite[Theorem~3.2,~Chapter~II]{Helgason} can be applied to the map $\Mbb(\B)_0 F \mapsto F^{-1}(0)$ to conclude that it is a homeomorphism. The main point is to replace $F$ with $F^{-1}$ or, in the case of a general Lie group $G$, an element $g$ with its inverse $g^{-1}$ (which is the notation of \cite{Helgason}). Because such assignment $F \mapsto F^{-1}$ is a diffeomorphism, the proof carries over without any essential change to conclude that $\Mbb(\B)_0 F \mapsto F^{-1}(0)$ defines indeed a homeomorphism.
	
	We now make use of \cite[Proposition~4.3,~Chapter~II]{Helgason}. Again we replace $F$ with $F^{-1}$ in the proof of the latter to consider the map $\Mbb(\B)_0 F \mapsto F^{-1}(0)$. As above and for the same reason, we can now conclude that such map is a diffeomorphism. Note that our modified version of \cite[Proposition~4.3,~Chapter~II]{Helgason} requires the map to be a homeomorphism (see the relevant statement in \cite[Chapter~II]{Helgason}), but this has been proved to hold in the previous paragraph.
	
	Collecting the claims from the two previous paragraphs, we conclude that the map $\Mbb(\B)_0 F \mapsto F^{-1}(0)$ yields a diffeomorphism $\Mbb(\B)_0 \backslash \Mbb(\B) \simeq \B$ as stated.
\end{proof}

Propositions~\ref{prop:AtoFAhomomorphism} and \ref{prop:BasM(B)quotient} provide successive quotients: $\Mbb(\B)$ as a quotient of $\Spe(1,1)$ and $\B$ as a quotient of $\Mbb(\B)$, respectively. This suggest that we should be able to write $\B$ as a quotient of $\Spe(1,1)$. Such is the content of the next result. Note that \cite[Theorem~4.2,~Chapter~II]{Helgason} implies that $(\Spe(1) \times \Spe(1))\backslash\Spe(1,1)$ admits a natural manifold structure. Once this has been observed, the statement below follows from Theorem~3.2 and Proposition~4.3 in \cite[Chapter~II]{Helgason} as a consequence of the transitivity of the $\Spe(1,1)$-action on $\B$ and the fact that this action has stabilizer at the origin realized by $\Spe(1) \times \Spe(1)$.

\begin{proposition}\label{prop:BasSp(1,1)quotient}
	The map given by
	\begin{align}\label{eq:BasSp(1,1)quotient}
		(\Spe(1) \times \Spe(1))\backslash\Spe(1,1)
			&\rightarrow \B  \\
			(\Spe(1) \times \Spe(1)) A 
			&\mapsto F_A(0), \notag
	\end{align}
	is a well defined diffeomorphism. 
\end{proposition}

\begin{remark}\label{rmk:BasSp(1,1)quotient}
	The diffeomorphism \eqref{eq:BasSp(1,1)quotient} is the usual one obtained from the $\Spe(1,1)$-action by classical M\"obius transformations. This corresponds to case CII, for $p=q=1$, in Table~V from \cite[pp.~518]{Helgason}. However, as noted above, such diffeomorphism can also be thought as coming from the quotients considered in Propositions~\ref{prop:AtoFAhomomorphism} and \ref{prop:BasM(B)quotient}. We now explain how to obtain \eqref{eq:BasSp(1,1)quotient} from these results. 
	
	First, by Proposition~\ref{prop:AtoFAhomomorphism} we have the isomorphism of Lie group
	\begin{align*}
		\Spe(1,1)/\Z_2 I_2 &\rightarrow \Mbb(\B)  \\
			A \Z_2 I_2 &\mapsto F_{A^{-1}},
	\end{align*}
	with respect to which Proposition~\ref{prop:M(B)action} shows that the subgroup corresponding to $\Mbb(\B)_0$ is $(\Spe(1) \times \Spe(1))/\Z_2 I_2$. If we now apply Proposition~\ref{prop:BasM(B)quotient}, then we obtain the diffeomorphism
	\begin{align}
		&\big((\Spe(1) \times \Spe(1))/\Z_2 I_2\big) \backslash
			\big(\Spe(1,1)/\Z_2 I_2\big) \rightarrow
		\Mbb(\B)_0 \backslash \Mbb(\B) \rightarrow \B  
				\label{eq:Sp(1,1)-Bdoublequotient-classical-1}	\\
		&\big((\Spe(1) \times \Spe(1))/\Z_2 I_2\big)\big(A \Z_2 I_2\big)
			\mapsto \Mbb(\B)_0 F_{A^{-1}} \mapsto 
				\big(F_{A^{-1}}\big)^{-1}(0) = F_A(0). \notag
	\end{align}
	On the other hand, the smooth right $\Spe(1,1)$-action on $(\Spe(1) \times \Spe(1))\backslash\Spe(1,1)$ (see \cite[Theorem~4.2,~Chapter~II]{Helgason}) factors to a smooth $\Spe(1,1)/\Z_2 I_2$-action on the same manifold. The reason is that $\Z_2 I_2$ is a central subgroup of $\Spe(1) \times \Spe(1)$, and so acts trivially on $(\Spe(1) \times \Spe(1))\backslash\Spe(1,1)$. Thus, we can apply Theorem~3.2 and Proposition~4.3 from \cite[Chapter~II]{Helgason} to conclude that the map
	\begin{align*}
		\big((\Spe(1) \times \Spe(1))/\Z_2 I_2\big) \backslash
			\big(\Spe(1,1)/\Z_2 I_2\big) &\rightarrow
		(\Spe(1) \times \Spe(1))\backslash\Spe(1,1)  \\
		\big((\Spe(1) \times \Spe(1))/\Z_2 I_2\big)(A \Z_2 I_2)
			&\mapsto (\Spe(1) \times \Spe(1))A
	\end{align*}
	is a diffeomorphism. Hence, \eqref{eq:Sp(1,1)-Bdoublequotient-classical-1} allows us to rewrite \eqref{eq:BasSp(1,1)quotient} as the diffeomorphism
	\begin{alignat}{2}
		(\Spe(1) \times \Spe(1))\backslash\Spe(1,1) &\rightarrow
			\Mbb(\B)_0 \backslash \Mbb(\B) &&\rightarrow \B 
				\label{eq:Sp(1,1)-Bdoublequotient-classical-2}  \\
		(\Spe(1) \times \Spe(1))A 
			&\mapsto \Mbb(\B)_0 F_{A^{-1}} &&\mapsto 
			\big(F_{A^{-1}}\big)^{-1}(0) = F_A(0).  \notag 
	\end{alignat}
	
	There are a couple of important observations derived from these computations. In the first place, and as claimed, we have shown that the canonical diffeomorphism \eqref{eq:BasSp(1,1)quotient} may be alternatively obtained through the two steps given by Propositions~\ref{prop:AtoFAhomomorphism} and \ref{prop:BasM(B)quotient}. This involves a non-standard construction because of the inverse $A^{-1}$ that appears in \eqref{eq:Sp(1,1)-Bdoublequotient-classical-2}. However, this sort of alternative choice will be better suited for our purposes latter on.
 	
 	On the other hand, the diffeomorphism \eqref{eq:BasSp(1,1)quotient} yields the usual Riemannian symmetric space structure of $\B$ that realizes the $1$-dimensional quaternionic hyperbolic and the $4$-dimensional real hyperbolic geometries. We refer to \cite{Helgason} for definitions and details (see also \cite{BisiStoppatoMobius,CaoParkerWang,ChenGreenberg}), but we shall now provide the Lie theoretic explanation of this fact.
	
	We have two natural smooth $\Spe(1,1)$-actions: one on $(\Spe(1) \times \Spe(1))\backslash \Spe(1,1)$ and the other on $\B$. For every $B \in \Spe(1,1)$, these actions give the transformations
	\begin{align*}
		(\Spe(1) \times \Spe(1))\backslash\Spe(1,1) &\rightarrow 
				(\Spe(1) \times \Spe(1))\backslash\Spe(1,1) &
			\B &\rightarrow \B  \\
		(\Spe(1) \times \Spe(1))A &\mapsto (\Spe(1) \times \Spe(1))AB &
			q &\mapsto F_B(q).
	\end{align*}
	The computation
	\[
		(\Spe(1) \times \Spe(1))AB \mapsto F_{AB}(0) 
			= F_B\big(F_A(0)\big)
	\]
	shows that, with respect to the $\Spe(1,1)$-actions above, the diffeomorphism \eqref{eq:BasSp(1,1)quotient} is $\Spe(1,1)$-equivariant. It is from this equivariance, together with the algebraic properties of the simple Lie group $\Spe(1,1)$, that the hyperbolic geometry on $\B$ is built. We refer to \cite[Proposition~3.4,~Chapter~IV]{Helgason} for further details. The latter introduces the construction of Riemannian structures on suitable quotients of (semi-)simple Lie groups by maximal compact subgroups, which includes our present case.
\end{remark}

\section{Analytic and geometric preliminaries}
\label{sec:RegularDiffGeom}
\subsection{Slice regularity}\label{subsec:sliceregularity}
It is well known that the appropriate notion of regularity for functions of a quaternionic variable allows to build very useful function theory and functional calculus. We will make use of the regularity notion introduced in \cite{GentiliStruppa2007} (see also~\cite{ColomboSabadiniStruppaFunctionalBook,ColomboSabadiniStruppaEntire,GentiliStoppatoStruppa2ndEd}), whose definition we now recall for our setup. In the rest of this work, we will denote by $\Sbb$ the $2$-dimensional sphere of elements $I \in \HH$ that satisfy $I^2 = -1$. These are the so-called quaternionic imaginary units.

\begin{definition}\label{def:sliceregular}
	A function $f : \B \rightarrow \HH$ is called (left) slice regular if, for every $I \in \Sbb$, we have
	\[
		\frac{1}{2} \bigg(
			\frac{\partial}{\partial x} + I \frac{\partial}{\partial y}
		\bigg) f(x + Iy) = 0,
	\]
	whenever $x,y \in \R$ are such that $x + yI \in \B$.
\end{definition}

We recall that the product, with respect to the operations in $\HH$, of slice regular functions is not necessarily slice regular. The same happens for the reciprocal of a function. However, it is possible to introduce a new product operation on the family of slice regular functions so as to obtain a real algebra with reciprocals, after restricting the domain on the complement of a suitable zero set. We now sketch the relevant basic facts and refer to \cite{ColomboSabadiniStruppaFunctionalBook,ColomboSabadiniStruppaEntire,GentiliStoppatoStruppa2ndEd} for further details. Note that we will restrict our discussion to the case of functions defined in $\B$.

Every slice regular function on $\B$ admits a power series expansion with right coefficients that converges uniformly on compact subsets of $\B$. If $f, g : \B \rightarrow \HH$ are slice regular with power series expansions
\[
	f(q) = \sum_{n=0}^\infty q^n a_n, \quad
	g(q) = \sum_{n=0}^\infty q^n b_n,
\]
then, its \textbf{(regular) $*$-product} is defined to be the function $f*g : \B \rightarrow \HH$ given by 
\[
	(f*g)(q) = \sum_{n=0}^\infty q^n c_n,
\]
where the coefficients $c_n$ are 
\[
	c_n = \sum_{k=0}^n a_k b_{n-k}.
\]
Then, $f*g$ is a well defined slice regular function on $\B$. Furthermore, the family of slice regular functions with the $*$-product and the usual pointwise addition turns out to be an algebra over $\R$. For $f$ as above, we introduce the \textbf{regular conjugate} as the function given by
\[
	f^c(q) = \sum_{n=0}^\infty q^n \overline{a}_n,
\]
for every $q \in \B$. The \textbf{symmetrization} of $f$ is now defined by
\[
	f^s = f * f^c = f^c * f,
\]
where the second identity is easily seen to hold. With these constructions at hand, the \textbf{regular reciprocal} or \textbf{$*$-reciprocal} of $f$ is the function $f^{-*}$ given by
\[
	f^{-*} = \frac{1}{f^s} f^c,
\]
which is defined on the complement in $\B$ of the zero set of $f^s$. The function $f^{-*}$ is slice regular and yields a reciprocal for $f$, with respect to the $*$-product, in the sense that
\[
	f^{-*} * f = f * f^{-*} = 1,
\]
on the complement of the zero set of $f^s$.

Using the $*$-product and the $*$-reciprocal, it was introduced in \cite{StoppatoMobius} a regular counterpart for classical M\"obius transformations that we now recall.

\begin{definition}\label{def:regularMobius}
	If $a, b \in \HH$ are given, then we will denote $l_{(a,b)}(q) = qa + b$, which defines a slice regular function. For every $A \in \Spe(1,1)$ of the form
	\[
		A = 
		\begin{pmatrix}
			a & c \\
			b & d
		\end{pmatrix},
	\]
	the \textbf{(slice) regular M\"obius transformation} associated to $A$ is defined by
	\begin{align*}
		\cF_A : \B &\longrightarrow \B  \\
			\cF_A &= l_{(c,d)}^{-*} * l_{(a,b)},
	\end{align*}
	which, for simplicity, is written as
	\[
		\cF_A(q) = (qc + d)^{-*}*(qa + b).
	\]
	We will also denote $\cM(\B) = \{ \cF_A : A \in \Spe(1,1)\}$, the family of all (slice) regular M\"obius transformations.
\end{definition}

The fundamental properties of regular M\"obius transformations have been developed in \cite{ColomboSabadiniStruppaFunctionalBook,GentiliStoppatoStruppa2ndEd,StoppatoMobius}. We will use the next result, which is a consequence of the theory found in \cite[section~9.4]{GentiliStoppatoStruppa2ndEd}.

\begin{proposition}[\cite{GentiliStoppatoStruppa2ndEd}]
	\label{prop:regularMobius}
	For every $A \in \Spe(1,1)$, the map $\cF_A : \B \rightarrow \B$ is a well defined regular bijection as well as a homeomorphism. Furthermore, $\cM(\B)$ is precisely the family of all slice regular bijections mapping $\B$ onto $\B$.
\end{proposition}

\begin{remark}\label{rmk:regMobius-composition}
	An important feature to take into account when dealing with slice regular functions is that their compositions are, in general, no longer slice regular. The simplest of examples can be obtained from the functions $q \mapsto q^2$ and $q \mapsto qa$, when $a \in \HH\setminus \R$. Similarly, $\cM(\B)$ is not closed under the composition of maps. To see this, let $-1 < a < 1$ and $u \in \Spe(1)$ be given, and consider the classical M\"obius transformations $f,g \in \Mbb(\B)$ defined~by
	\[
		f(q) = (1 - qa)^{-1} (q - a), \quad
		g(q) = qu.
	\]
	In particular, both are bijections of $\B$. It is easy to see that both $f,g$ are slice regular and so Proposition~\ref{prop:regularMobius} implies that $f,g \in \cM(\B)$. The choice of $a$ as real number allows to compute
	\[
		(f\circ g) (q) = \sum_{n=0}^\infty (qu)^n a^n (qu - a)
			= -a + \sum_{n=1}^\infty (qu)^n a^{n-1} (1 - a^2),
	\]
	for every $q \in \B$. Hence, we conclude that $f \circ g$ is slice regular only for $u = \pm 1$. It follows that $f \circ g \not\in \cM(\B)$ whenever we choose $u \in \Spe(1) \setminus \{\pm 1\}$. Hence, $\cM(\B)$ is not a group under the composition of maps.
	
	Furthermore, the assignment $A \mapsto \cF_A$, where $A \in \Spe(1,1)$, does not satisfy the analogue of \eqref{eq:FAB}. As a matter of fact, the previous example proves the existence of $A, B \in \Spe(1,1)$ such that $\cF_B \circ \cF_A \not \in \cM(\B)$, and so $\cF_B \circ \cF_A \not= \cF_C$ for every $C \in \Spe(1,1)$.
	
	On the other hand, it is important to recall that there is a quite general setup where one can ensure the slice regularity of the composition $f\circ g$ of two slice regular functions $f,g : \B \rightarrow \B$. The slice regularity of $f \circ g$ holds whenever $g$ is a so-called real function or quaternionic intrinsic function, i.e.~when $g$ admits a power series expansion with real coefficients (see~\cite[Sections~2.4~and~2.5]{ColomboSabadiniStruppaEntire} for further details).
\end{remark}

The next result is a consequence of the proof of \cite[Proposition~9.10]{GentiliStoppatoStruppa2ndEd}. Note that the latter is stated for slice regular M\"obius transformations defined over $\HH$ (minus a zero set) and the following considers those over $\B$. However, the same arguments used in \cite{GentiliStoppatoStruppa2ndEd} applies for our case. Also, observe that 
\[
	\HH I_2 \cap \Spe(1,1) = \Spe(1)I_2,
\] 
and this allows us to assume that $u \in \Spe(1)$ in the statement below from the theory developed in \cite{GentiliStoppatoStruppa2ndEd}.

\begin{proposition}[\cite{GentiliStoppatoStruppa2ndEd}]
	\label{prop:A=uB}
	For $A, B \in \Spe(1,1)$, we have $\cF_A = \cF_B$ if and only if there exists $u \in \Spe(1)$ such that $A = u B$.
\end{proposition}

\subsection{Principal fiber bundles}
\label{subsec:principalbundles}
The geometric constructions that we will apply make use of some well-known notation and results from differential geometry. We recollect here those that will be needed for the rest of this work and refer to \cite{KNvol1,Sharpe} for further details and proofs.

\begin{definition}\label{def:principalbundle}
	A principal fiber bundle with $G$ as structure Lie group is an ordered collection $(\pi, P, M)$, where $P$ and $M$ are smooth manifolds and $\pi : P \rightarrow M$ is a map. The following conditions are also required to be satisfied.
	\begin{enumerate}
		\item The map $\pi$ is a submersion, i.e.~a surjective smooth map whose differential at every point is surjective.
		\item There is a free and proper smooth right $G$-action on $P$ whose orbits are precisely the fibers of $\pi$.
		\item For every $p \in M$, there is an open subset $U \subset M$ containing $p$ and a $G$-equivariant diffeomorphism $\varphi_U : U \times G \rightarrow \pi^{-1}(U)$, for the right $G$-action on the first factor for the domain, such that the diagram
		\[
		\xymatrix{
			U \times G \ar[rr]^-{\varphi_U} \ar[dr]^-{\pi_1} & 
						& \pi^{-1}(U) \ar[dl]_-{\pi} \\
				& U &
		}
		\]
		is commutative, where $\pi_1$ is the projection onto the first factor.
	\end{enumerate}
\end{definition}

	Note that we have used right actions to define principal fiber bundles. However, there is an obvious corresponding notion for the case of left actions. We will have the opportunity to use both.

	For simplicity, one usually identifies a principal fiber bundle, as in Definition~\ref{def:principalbundle}, with the map $\pi$ from $P$ onto $M$. In particular, we will speak of the principal fiber bundle $P \rightarrow M$ with structure group $G$, where the map from $P$ onto $M$ and the $G$-action have been specified and satisfy the conditions from Definition~\ref{def:principalbundle}. In this situation, the manifolds $P$ and $M$ are called the total and base spaces of the principal fiber bundle, respectively.
	
	We recall that an action is free when the identity element is the only one that has a fixed point. Also, a smooth $G$-action on a manifold $P$ is called proper if, for every compact subset $K \subset P$, the set
	\[
		\{g \in G : gK \cap K \not=\emptyset\},
	\]
	is compact in $G$. In particular, every smooth action of a compact Lie group is proper.

	Condition (3) from Definition~\ref{def:principalbundle} is referred by saying that $\pi : P \rightarrow M$ is locally trivial as a principal bundle. This property entitles the existence of local sections as stated in the next result.
	
	\begin{proposition}\label{prop:localsections}
		Let $\pi : P \rightarrow M$ be a principal fiber bundle with structure group $G$. Then, for every $p \in M$ there exists an open neighborhood $U \subset M$ of $p$ and a smooth map $\sigma : U \rightarrow P$ such that $\pi \circ \sigma = Id_U$, the identity map in $U$. The map $\sigma$ is called a local section of $\pi$.
	\end{proposition}
	
	Finally, we have the next result that allows us to construct principal fiber bundles from smooth free proper actions. For its proof we refer to \cite[Appendix~E]{Sharpe}.
	
	\begin{proposition}\label{prop:buildprincipalbundle}
		Let $P$ be a smooth manifold that admits a smooth right $G$-action, where $G$ is a Lie group. If the $G$-action on $P$ is free and proper, then the following holds.
		\begin{enumerate}
			\item The space $M = P/G$ with the quotient topology is metrizable Hausdorff and admits a unique manifold structure such that the quotient map $\pi : P \rightarrow M$ is a submersion.
			\item The family $(\pi, P, M)$ is a principal fiber bundle with structure group $G$.
		\end{enumerate}
		Furthermore, we have $\dim M = \dim P - \dim G$.
	\end{proposition}

We observe that there is an obvious corresponding result for left $G$-actions, where one now considers $M = G\backslash P$.

A particular example that will be useful can be described as follows. Let $G$ be a Lie group and $H$ a closed subgroup (hence a Lie subgroup). By the smoothness of the product of $G$, both the left and right actions of $H$ on $G$ are smooth. These actions are clearly free, and it is easy to prove that they are also proper as a consequence of the fact that $H$ is closed. Hence, the quotients $G/H$ and $H\backslash G$ both have manifold structures with the properties that follow from Proposition~\ref{prop:buildprincipalbundle}.

\section{Slice regular M\"obius transformations}
\subsection{Slice regular M\"obius transformations as a manifold}
Our first main result proves the existence of a manifold structure on $\cM(\B)$ induced from the Lie group structure of $\Spe(1,1)$.

\begin{theorem}
	\label{thm:cM(B)manifold}
	The mapping 
	\begin{align*}
		\Spe(1,1)/\Spe(1)I_2 &\rightarrow \cM(\B)  \\
			A\Spe(1)I_2 &\mapsto \cF_{A^{-1}},
	\end{align*}
	is a well defined bijection of sets. In particular, $\cM(\B)$ has a unique manifold structure for which this bijection is a diffeomorphism. Furthermore, $\cM(\B)$ has dimension $7$ as a manifold for this structure.
\end{theorem}
\begin{proof}
	By Proposition~\ref{prop:A=uB}, for $A, B \in \Spe(1,1)$ we have $\cF_{A^{-1}} = \cF_{B^{-1}}$ if and only if there exists $u \in \Spe(1)$ such that $A = (A^{-1})^{-1} = (u B^{-1})^{-1} = B u^{-1}$. This shows that the mapping is indeed a well defined bijection. Since $\Spe(1) I_2$ is a closed subgroup of $\Spe(1,1)$, by the discussion at the end of subsection~\ref{subsec:principalbundles}, Proposition~\ref{prop:buildprincipalbundle} applies and the quotient $\Spe(1,1)/\Spe(1)I_2$ admits a unique manifold structure
	so that the quotient map from $\Spe(1,1)$ is a submersion. This yields on $\cM(\B)$ a corresponding manifold structure through the given bijection. Finally, by construction, the manifold $\cM(\B)$ satisfies
	\[
		\dim \cM(\B) = \dim \Spe(1,1)/\Spe(1)I_2
			= \dim \Spe(1,1) - \dim \Spe(1)
			= 10 - 3 = 7,
	\]
	thus completing the proof.
\end{proof}

From now on, we will consider $\cM(\B)$ with the manifold structure obtained from Theorem~\ref{thm:cM(B)manifold}. The proof of the latter, which uses Proposition~\ref{prop:buildprincipalbundle}, implies the next result.

\begin{corollary}\label{cor:Sp(1,1)cM(B)PrincipalFiberBundle}
	The assignment given by
	\begin{align*}
		\Spe(1,1) &\rightarrow \cM(\B)  \\
			A &\mapsto \cF_{A^{-1}},
	\end{align*}
	yields a principal fiber bundle with total space $\Spe(1,1)$, base space $\cM(\B)$ and structure group $\Spe(1)$, where the latter acts on $\Spe(1,1)$ by right translations.
\end{corollary}

\begin{remark}\label{rmk:cM(B)manifold}
	It is useful to compare Proposition~\ref{prop:AtoFAhomomorphism} and Theorem~\ref{thm:cM(B)manifold}. The former yields a Lie group structure on $\Mbb(\B)$ and the latter a manifold structure on $\cM(\B)$, both cases realize them as quotients of the Lie group $\Spe(1,1)$ by a closed subgroup. However, for $\Mbb(\B)$ the subgroup is $\Z_2 I_2$, which is central and thus normal, and for $\cM(\B)$ the subgroup is $\Spe(1)I_2$ which is not normal. For example, we have
	\[
		\begin{pmatrix*}
			a & 0 \\
			0 & 1 
		\end{pmatrix*}
		\begin{pmatrix*}
			u & 0 \\
			0 & u 
		\end{pmatrix*}
		\begin{pmatrix*}
			a & 0 \\
			0 & 1
		\end{pmatrix*}^{-1} =
		\begin{pmatrix*}[c]
			aua^{-1} & 0 \\
			0 & u
		\end{pmatrix*}
	\]
	where $a, u \in \Spe(1)$, and this matrix does not belong to $\Spe(1) I_2$ unless $au = ua$. Hence, that $\Spe(1) I_2$ is not normal follows from the non-commutativity of $\HH$. This implies the lack of a group structure on $\cM(\B)$ in the constructions above.
	
	On the other hand, Corollary~\ref{cor:Sp(1,1)cM(B)PrincipalFiberBundle} proves the canonical nature of the manifold structure on $\cM(\B)$ introduced through Theorem~\ref{thm:cM(B)manifold}. This is the case because the assignment $A \mapsto \cF_A$ canonically yields the set $\cM(\B)$. Note that using either $\cF_A$ or $\cF_{A^{-1}}$ yields the same set (since $\Spe(1,1)$ is a group) and the same manifold structure (since $A \mapsto A^{-1}$ is diffeomorphism of $\Spe(1,1)$). As found below, it will be more convenient to use the expressions considered in Theorem~\ref{thm:cM(B)manifold} and Corollary~\ref{cor:Sp(1,1)cM(B)PrincipalFiberBundle}.
\end{remark}

The previous results and observations lead us to the next result.

\begin{corollary}\label{cor:cM(B)notgroup}
	The set $\cM(\B)$ of regular M\"obius transformations on $\B$ does not admit a group structure such that the assignment $A \mapsto \cF_A$ from $\Spe(1,1)$ onto $\cM(\B)$ is a either a homomorphism or an anti-homomorphism.
\end{corollary}
\begin{proof}
	Suppose that such group structure on $\cM(\B)$ does exist. Remark~\ref{rmk:AtoFAhomomorphism} allows us to assume that the assignment $\Spe(1,1) \rightarrow \cM(\B)$ given by $A \mapsto \cF_A$ is a homomorphism. By Theorem~\ref{thm:cM(B)manifold}, or even directly from Proposition~\ref{prop:A=uB}, every fiber of this map is of the form $A_0\Spe(1)$, for some $A_0 \in \Spe(1,1)$, which is a subgroup if and only if $A_0 = u_0 I_2$, for some $u_0 \in \Spe(1)$. This implies that the kernel of the said homomorphism is necessarily $\Spe(1) I_2$, which is not normal as noted in Remark~\ref{rmk:cM(B)manifold}. This contradiction completes the proof.
\end{proof}

\begin{remark}\label{rmk:cM(B)notgroup}
	It is important to recall that \cite[Section~2.5]{ColomboSabadiniStruppaEntire} has introduced a notion of regular composition $f \bullet g$ of two slice regular functions $f,g$. One is bounded to ask whether this would yield a group structure on $\cM(\B)$. However, as noted in \cite[Proposition~2.8]{ColomboSabadiniStruppaEntire}, this operation is in general not associative. Nevertheless, the composition $\bullet$ could provide some sort of algebraic structure on $\cM(\B)$ since it is naturally built from the $*$-product. Further research in this direction could provide interesting results.
	
	On the other hand, regular M\"obius transformations yield the building blocks of the so-called (slice regular) Blaschke products. We refer to \cite[Section~6.3]{AlpayColomboSabadiniSliceSchur} and \cite[Section~4.4]{ColomboSabadiniStruppaEntire} for definitions and further details. This yields a product of elements in $\cM(\B)$ whose study is fundamental to hyperholomorphic analysis. Even though Blaschke products do not yield a group structure on $\cM(\B)$, we believe they might be relevant to the study of the group-like structures studied in this work. 
	
	Finally, we observe that \cite{Vlacci} has introduced another notion of regular composition for slice regular functions. This operation does not seem to entitle a group structure on $\cM(\B)$. However, it is also of interest to determine the precise relationship of this operation with group-like structures.
\end{remark}

The next step is to identify the manifold structure of $\cM(\B)$ given by Theorem~\ref{thm:cM(B)manifold}. In other words, realize $\cM(\B)$ as a concrete manifold. To achieve this, we recall a result that describes all regular M\"obius transformations on $\B$ as a set. This result is a restatement of \cite[Corollary~9.17]{GentiliStoppatoStruppa2ndEd} and \cite[Corollary~1.8]{StoppatoMobius}.

\begin{proposition}[\cite{GentiliStoppatoStruppa2ndEd,StoppatoMobius}]
	\label{prop:cM(B)isBxSp(1)assets}
	Let us denote by $\varphi : \B \times \Spe(1) \rightarrow \cM(\B)$ the map that assigns to every $(a,u) \in \B \times \Spe(1)$ the regular M\"obius transformation represented~by 
	\[
		(1 - q\overline{a})^{-*} *(q - a)u.
	\]
	Then, $\varphi$ is a bijection of sets.
\end{proposition}

It turns out that $\cM(\B)$ can be precisely realized by $\B \times \Spe(1)$ as a manifold, not just as a~set, as we prove in the next result.

\begin{theorem}\label{thm:cM(B)isBxSp(1)asmanifolds}
	The map $\varphi : \B \times \Spe(1) \rightarrow \cM(\B)$, defined in Proposition~\ref{prop:cM(B)isBxSp(1)assets}, is a diffeomorphism for the manifold structure on $\cM(\B)$ given by Theorem~\ref{thm:cM(B)manifold}. In particular, $\cM(\B)$ is diffeomorphic to $\R^4 \times S^3$, where $S^3$ denotes the $3$-dimensional sphere.
\end{theorem}
\begin{proof}
	Let us consider the diagram
	\[
	\xymatrixcolsep{4pc}
	\xymatrix{
		& \Spe(1,1) 
		\ar[d]^-{\pi}	\\
		\B \times \Spe(1) \ar[r]^-{\varphi} 
			\ar[ur]^-{\widetilde{\varphi}}
		& \cM(\B)
	}
	\]
	where $\pi(A) = \cF_{A^{-1}}$, for every $A \in \Spe(1,1)$, and where we also take
	\begin{align*}
		\widetilde{\varphi}(a,u) &=
			(1 - |a|^2)^{-\frac{1}{2}}
				\begin{pmatrix}
					\overline{u} & \overline{au}  \\
					a & 1
				\end{pmatrix}  \\
			&=
			\Bigg(
				(1 - |a|^2)^{-\frac{1}{2}}
				\begin{pmatrix*}[r]
					1 & -\overline{a}  \\
					-a & 1
				\end{pmatrix*}
				\begin{pmatrix}
					u & 0 \\
					0 & 1
				\end{pmatrix}
			\Bigg)^{-1}
	\end{align*}
	for every $(a,u) \in \B \times \Spe(1)$. 
	
	Recall that $\Spe(1,1)$, as a manifold, is a submanifold of the space of $2 \times 2$ matrices with quaternionic entries. Note that the latter is a $16$-dimensional real vector space.
	
	The map $\widetilde{\varphi}$ is well defined in the sense that $\widetilde{\varphi}(a,u)$ belongs to $\Spe(1,1)$, for every $(a,u) \in \B \times \Spe(1)$. This can be verified by a direct computation. Once this has been observed, we note that $\widetilde{\varphi}$ is a smooth map; the reason is that $\Spe(1,1)$ is a closed embedded submanifold of the space of matrices referred above, so that \cite[Theorem~1.32]{WarnerBook} can be applied to ensure the smoothness of $\widetilde{\varphi}$ when considered as mapping into $\Spe(1,1)$.
	
	The map $\pi$ is the projection of the principal fiber bundle considered in Corollary~\ref{cor:Sp(1,1)cM(B)PrincipalFiberBundle}, which implies that $\pi$ is a submersion: i.e.~a smooth surjective map with surjective differential at every point.
	
	A straightforward computation shows that the diagram above is commutative. In other words, we have $\varphi = \pi \circ \widetilde{\varphi}$, and so $\varphi$ is smooth. By Proposition~\ref{prop:cM(B)isBxSp(1)assets}, the map $\varphi^{-1}$ exists and it remains to prove that it is smooth as well. Since both $\B \times \Spe(1)$ and $\cM(\B)$ have dimension $7$, by the Theorem of the Inverse Function for manifolds, it is enough to show that $\dif \varphi_{(a_0,u_0)}$ is an isomorphism at every point $(a_0,u_0) \in \B \times \Spe(1)$. 
	
	Let us fix $(a_0,u_0) \in \B \times \Spe(1)$ and let us start by computing $\dif \widetilde{\varphi}_{(a_0,u_0)}$. We observe that
	\[
		\widetilde{\varphi}(a,u) 
			= \widetilde{\varphi}(0,u) \widetilde{\varphi}(a,1)
	\]
	for every $(a,u) \in \B \times \Spe(1)$. Since the product of matrices is bilinear, we can apply a Leibniz type rule to obtain, for every $(b,v) \in T_{(a_0,u_0)} \B \times \Spe(1)$
	\begin{align} \label{eq:diftildevarphi1}
		\dif \widetilde{\varphi}_{(a_0,u_0)}(b,v)
			&= \dif \widetilde{\varphi}_{(a_0,u_0)}(b,0)
				+ \dif \widetilde{\varphi}_{(a_0,u_0)}(0,v)  \\
			&= \widetilde{\varphi}(0,u_0) 
						\dif \widetilde{\varphi}_{(a_0,1)}(b,0)  \notag
				+ \dif \widetilde{\varphi}_{(0,u_0)}(0,v) 
						\widetilde{\varphi}(a_0,1).
	\end{align}
	We now proceed to compute the differentials in the second line of \eqref{eq:diftildevarphi1}. We clearly have
	\[
		\dif \widetilde{\varphi}_{(0,u_0)}(0,v) =
			\begin{pmatrix}
				\overline{v} & 0 \\
				0 & 0 
			\end{pmatrix},
	\]
	which yields the second differential. We will denote by $x \cdot y \in \R$ the usual inner product of $x, y \in \HH$, where we identify $\HH \simeq \R^4$. Then, a straightforward computation yields the first differential
	\begin{multline*}
		\dif \widetilde{\varphi}_{(a_0,1)}(b,0) = \\
			= (1 - |a_0|^2)^{-\frac{3}{2}}
				\begin{pmatrix}
					a_0\cdot b &
						(a_0\cdot b) \overline{a_0} 
							+ (1 - |a_0|^2)\overline{b} \\
					(a_0\cdot b) a_0 + (1 - |a_0|^2)b &
						a_0 \cdot b
				\end{pmatrix}.
	\end{multline*}
	Replacing the last formulas in \eqref{eq:diftildevarphi1} yields, after some manipulations, the formula
	\begin{multline}\label{eq:diftildevarphi2}
		\dif \widetilde{\varphi}_{(a_0,u_0)}(b,v) =
			 (1 - |a_0|^2)^{-\frac{3}{2}} \times \\
			 	\times
				\begin{pmatrix}
					(a_0 \cdot b)\overline{u} + (1 - |a_0|^2) \overline{v} &
						(a_0\cdot b) \overline{a_0 u_0} 
								+ (1 - |a_0|^2)(\overline{b u_0 + a_0 v}) \\
					(a_0\cdot b) a_0 + (1 - |a_0|^2)b &
							a_0 \cdot b
				\end{pmatrix},
	\end{multline}
	for every $(b,v) \in T_{(a_0,u_0)} \B \times \Spe(1)$. Let us assume that this expression vanishes. Hence, we have $a_0 \cdot b = 0$. Thus, the vanishing of the first column clearly implies $(b,v) = (0,0)$. This proves that $\dif \widetilde{\varphi}_{(a_0,u_0)}$ is injective for every $(a_0,u_0) \in \B \times \Spe(1)$. 
	
	On the other hand, we note that for every $A \in \Spe(1,1)$, we have
	\[
		\ker(\dif \pi_A) = T_A A \Spe(1),
	\]
	the tangent space at $A$ of the $\Spe(1)$-orbit. This is a consequence of Corollary~\ref{cor:Sp(1,1)cM(B)PrincipalFiberBundle} and the properties of principal fiber bundles. By using the exponential map of $\Spe(1)$, we can compute
	\begin{align*}
		T_A A\Spe(1) 
			&= \{ Aw : w \in \spe(1) \}  \\
			&= \{ Aw : w \in \HH,\; w + \overline{w} = 0 \},
	\end{align*}
	where the last identity follows from the description of $\spe(1)$, the Lie algebra of $\Spe(1)$, as the purely imaginary quaternions. We conclude that the tangent space $T_{\widetilde{\varphi}(a,u)} \widetilde{\varphi}(a,u) \Spe(1)$ consists of matrices of the form
	\[
		(1 - |a|^2)^{-\frac{1}{2}}
			\begin{pmatrix}
				\overline{u}w & \overline{au} w  \\
				aw & w
			\end{pmatrix}
	\]
	where $w \in \HH$ is purely imaginary. By comparing the entry at position $(2,2)$ in the last expression with the corresponding one of \eqref{eq:diftildevarphi2} we conclude that
	\begin{equation}\label{eq:cap=0}
		\dif\widetilde{\varphi}_{(a_0,u_0)}
			\big(T_{(a_0,u_0)} \B \times \Spe(1)\big) \cap 
				\ker\big(\dif \pi_{\widetilde{\varphi}(a_0,u_0)}\big) = 0,
	\end{equation}
	for every $(a_0,u_0) \in \B \times \Spe(1)$, because the inner product $a_0\cdot b$ is real and $w$ is purely imaginary in those expressions. In other words, at every point, the image of the differential of $\widetilde{\varphi}$ has intersection $0$ with the kernel of the differential of $\pi$.
	
	Applying the chain rule to the commutative diagram considered above, we obtain
	\[
		\dif \varphi_{(a_0,u_0)} = 
			\dif \pi_{\widetilde{\varphi}(a_0,u_0)} \circ
				\dif \widetilde{\varphi}_{(a_0,u_0)},
	\]
	for every $(a_0,u_0) \in \B \times \Spe(1)$. Thus, the injectivity of $\dif \widetilde{\varphi}_{(a_0,u_0)}$ established above and \eqref{eq:cap=0} imply that $\dif \varphi_{(a_0,u_0)}$ is injective, and so it is an isomorphism because $\B \times \Spe(1)$ and $\cM(\B)$ have the same dimension. This holds at every point $(a_0,u_0)$ of $\B \times \Spe(1)$ and so, as noted above, we conclude that $\varphi$ is a diffeomorphism.
	
	The last claim is now obvious, because $\B$ is diffeomorphic to $\R^4$ and $\Spe(1)$ is precisely the $3$-dimensional sphere contained in $\HH \simeq \R^4$.
\end{proof}

\begin{remark}\label{rmk:cM(B)isBxSp(1)asmanifolds}
	From the previous results, we have observed two natural set descriptions of the family $\cM(\B)$ of regular M\"obius transformations. These are given by the following bijections
	\begin{align*}
		\Spe(1,1)/\Spe(1) I_2 &\longleftrightarrow \cM(\B) &
		\B \times \Spe(1) &\longleftrightarrow \cM(\B) \\
			A\Spe(1) &\longleftrightarrow \cF_{A^{-1}} &
			(a,u) &\longleftrightarrow 
						(1 - q\overline{a})^{-*}*(q - a)u.
	\end{align*}
	Either of them can be used to define a manifold structure. Basically, this is achieved by declaring the bijection to be a diffeomorphism, which is precisely what we have done in Theorem~\ref{thm:cM(B)manifold}. The main point of Theorem~\ref{thm:cM(B)isBxSp(1)asmanifolds} is to prove that both bijections define the same manifold structure on $\cM(\B)$.
	
	On the other hand, both correspondences yield interesting properties of the manifold structure on $\cM(\B)$. The first one above provides a direct link with the group structure of $\Spe(1,1)$ as Theorem~\ref{thm:cM(B)manifold} has shown, and it also provides a principal fiber bundle to work with. The second correspondence above relates the manifold structure with the usual description of all regular M\"obius transformations stated in Proposition~\ref{prop:cM(B)isBxSp(1)assets}.
\end{remark}

\subsection{Slice regular evaluation-action of $\cM(\B)$ and $\Spe(1,1)$ on $\B$} We start by studying properties of the evaluation map of $\cM(\B)$ on $\B$. Note that such map is the closest object that can be thought of as some sort of ``action'' of $\cM(\B)$ on~$\B$. For this reason, we will refer to it as the \textbf{evaluation-action of $\cM(\B)$ on $\B$}. In the rest of this work, for every $u \in \Spe(1)$, we will denote by $R_u$ the slice regular M\"obius transformation given by $R_u(q) = qu$, for every $q \in \B$.

\begin{theorem}\label{thm:cM(B)action}
	For the manifold structure on $\cM(\B)$ given in Theorem~\ref{thm:cM(B)manifold}, the evaluation-action map
	\begin{align*}
		\cM(\B) \times \B &\rightarrow \B  \\
			(\cF,q) &\mapsto \cF(q),
	\end{align*}
	is smooth. Furthermore, the following properties hold.
	\begin{enumerate}
		\item Transitivity: For every $p,q \in \B$ there exists $\cF \in \cM(\B)$ such that $\cF(p) = q$.
		\item Isotropy at $0$: The family of $\cM(\B)_0 = \{ \cF \in \cM(\B) : \cF(0) = 0 \}$ is given by
		\[
			\cM(\B)_0 = \{ R_u : u \in \Spe(1) \}.
		\]
		\item Structure of the isotropy at $0$: $\cM(\B)_0$ is a compact submanifold of $\cM(\B)$ and a Lie group for this manifold structure together with the composition of maps. Furthermore, the assignment
		\begin{align*}
			\rho : \Spe(1) &\rightarrow \cM(\B)_0  \\
				u &\mapsto R_{u^{-1}},
		\end{align*}
		is an isomorphism of Lie groups.
	\end{enumerate}
\end{theorem}
\begin{proof}
	The transitivity of regular M\"obius transformations has been proved before in \cite[Proposition~9.21]{GentiliStoppatoStruppa2ndEd}. Hence, claim (1) holds. Similarly, claim (2) follows from \cite[Lemma~9.19]{GentiliStoppatoStruppa2ndEd}.

	Let us now consider the commutative diagram 
	\[
		\xymatrixcolsep{8pc}
		\xymatrix{
			\B \times \Spe(1,1) \times \B 
			\ar[dd]_-{\varphi \times Id}
			\ar[rd]|-{(a,u,q)\mapsto \varphi(a,u)(q)} 
			& \\
			& \B \\
			\cM(\B) \times \B \ar[ru]|-{(\cF,q)\mapsto \cF(q)} &
		}
	\]
	where $\varphi$ is the diffeomorphism $\B \times \Spe(1) \rightarrow \cM(\B)$ introduced in Proposition~\ref{prop:cM(B)isBxSp(1)assets}. A straightforward computation shows that
	\begin{equation}\label{eq:varphi(a,u)(q)}
		\varphi(a,u)(q) =
			(1 - 2\re(a)q + q^2 |a|^2)^{-1}
				(q(a^2 + 1) - (q^2 + 1)a)u,
	\end{equation}
	which is clearly a smooth function of $a,q \in \B$ and $u \in \Spe(1)$. Note that similar computations can be found in \cite[Definition~6.3.3]{AlpayColomboSabadiniSliceSchur} and \cite[Definition~4.7]{ColomboSabadiniStruppaEntire}. Theorem~\ref{thm:cM(B)isBxSp(1)asmanifolds}, \eqref{eq:varphi(a,u)(q)} and the commutative diagram above show that the evaluation map of $\cM(\B)$ on $\B$ is indeed smooth.
	
	On the other hand, by (2) we have
	\[
		\varphi(\{0\} \times \Spe(1)) = \cM(\B)_0,
	\]
	and so $\cM(\B)_0$ is a compact submanifold of $\cM(\B)$. Furthermore, the following identities hold
	\[
		\varphi(1,uv) = R_{uv} = R_v \circ R_u, \quad
		\varphi(1,u^{-1}) = R_{u^{-1}} = \big(R_u\big)^{-1}.
	\]
	These identities together with (2) and the fact that $\Spe(1)$ is a Lie group show that the composition and inverse operations on $\cM(\B)_0$ are smooth. They also show that $\rho$, as defined in the statement, is an isomorphism of Lie groups. Hence, (3) holds and this completes the proof.
\end{proof}

We now prove that the Lie group $\cM(\B)_0$ has a natural smooth action on $\cM(\B)$.

\begin{proposition}\label{prop:cM(B)0actingoncM(B)}
	The map given by
	\begin{align*}
		\cM(\B)_0 \times \cM(\B) &\rightarrow \cM(\B)  \\
			(\cF_0, \cF) &\mapsto \cF_0 \circ \cF,
	\end{align*}
	defines a smooth action of the Lie group $\cM(\B)_0$ on the manifold $\cM(\B)$. Furthermore, this $\cM(\B)_0$-action is free and proper.
\end{proposition}
\begin{proof}
	Let $\cF_0 \in \cM(\B)_0$ and $\cF \in \cM(\B)$ be given. Then, Theorem~\ref{thm:cM(B)action} implies that there exists $u \in \Spe(1)$ such that $\cF_0 = R_u$, and so $\cF_0 \circ \cF = R_u \circ \cF$ which does belong to $\cM(\B)$. It follows that the map from the statement is well defined. Furthermore, it does yield an action by the associativity of the composition of maps.
	
	For the map $\varphi$ defined in Proposition~\ref{prop:cM(B)isBxSp(1)assets}, let us consider the diagram
	\[
		\xymatrixcolsep{8pc}
		\xymatrix{
			\Spe(1) \times \B \times \Spe(1)
			\ar[dd]_-{\substack{(v,a,u)\\\downmapsto\\(R_v,\varphi(a,u))}}
			\ar[r]^-{(v,a,u) \mapsto (a,uv)} 
			& \B \times \Spe(1) \ar[dd]^{\varphi} \\
				& \\
			\cM(\B)_0 \times \cM(\B) 
				\ar[r]_-{(\cF_0,\cF)\mapsto \cF_0 \circ \cF} 
				& \cM(\B)
	}
	\]
	which is easily seen to be commutative. The vertical arrows are diffeomorphisms by Theorems~\ref{thm:cM(B)isBxSp(1)asmanifolds} and \ref{thm:cM(B)action}(3), and the top horizontal arrow is smooth because $\Spe(1)$ is a Lie group. This implies the smoothness of the bottom arrow, thus showing that the $\cM(\B)_0$-action on $\cM(\B)$ is smooth.

	That the $\cM(\B)_0$-action is proper follows from the compactness of $\cM(\B)_0$. Let us now assume that we have
	\[
		R_u \circ \cF = \cF,
	\]
	for some $u \in \Spe(1)$ and $\cF \in \cM(\B)$. In other words, $R_u \in \cM(\B)_0$ fixes some element in $\cM(\B)$. If we choose $q_0 \in \B$ such that $q_1 = \cF(q_0) \not= 0$, then
	\[
		q_1 u = R_u \circ \cF(q_0) = \cF(q_0) = q_1,
	\]
	which implies that $u = 1$ and thus $R_u$ is the identity element of $\cM(\B)_0$. Hence, the $\cM(\B)_0$-action is free.
\end{proof}

Propositions~\ref{prop:cM(B)0actingoncM(B)} and \ref{prop:buildprincipalbundle} allow us to conclude the next result.

\begin{proposition}\label{prop:cM(B)0modcM(B)}
	The quotient $\cM(\B)_0 \backslash \cM(\B)$ admits a unique manifold structure such that the map $\cM(\B) \rightarrow \cM(\B)_0 \backslash \cM(\B)$ given by $\cF \mapsto \cM(\B)_0 \cF$ yields a principal fiber bundle. The structure group of this bundle is $\cM(\B)_0$ for the left action on $\cM(\B)$ defined in Proposition~\ref{prop:cM(B)0actingoncM(B)}.
\end{proposition}

\subsection{The unit ball $\B$ as a quotient: slice regular case}
We will use the previous results to express $\B$ as a quotient of $\cM(\B)$, in the sense of manifolds and principal fiber bundles. We start with the next result (see~\cite[Lemma~5.5]{QB-SliceKahler}) whose proof we present for the sake of completeness.

\begin{lemma}\label{lem:inverseorbitat0}
	Let $\cF_1, \cF_2 \in \cM(\B)$ be given. Then, the following conditions are equivalent.
	\begin{enumerate}
		\item $\cF_1^{-1}(0) = \cF_2^{-1}(0)$.
		\item There exists $\cF_0 \in \cM(\B)_0$ such that $\cF_1 = \cF_0 \circ \cF_2$.
	\end{enumerate}
\end{lemma}
\begin{proof}
	By the definition of $\cM(\B)_0$, it is clear that (2) implies (1). Hence, let us assume that (1) holds. Proposition~\ref{prop:cM(B)isBxSp(1)assets} implies the existence of $a_1, a_2 \in \B$ and $u_1, u_2 \in \Spe(1)$ such that
	\[
		\cF_1 = \varphi(a_1,u_1), \quad
		\cF_2 = \varphi(a_2,u_2).
	\]
	By assumption, $\cF_1$ and $\cF_2$ vanish at the same point, which implies that $a_1 = a_2$. The latter is a consequence of \eqref{eq:varphi(a,u)(q)} and the fact that both $\cF_1$ and $\cF_2$ are bijective. Now using again \eqref{eq:varphi(a,u)(q)} we conclude that
	\[
		\cF_1 = R_u \circ \cF_2,
	\]
	where $u = u_2^{-1} u_1$. Then, Theorem~\ref{thm:cM(B)action}(2) yields (2) and thus completes the proof.
\end{proof}

We now show that $\B$ can be seen as a quotient of $\cM(\B)$.

\begin{theorem}\label{thm:BascM(B)quotient}
	For the evaluation-action of $\cM(\B)$ on $\B$, the \textbf{inverse orbit map} $\cM(\B) \rightarrow \B$ at $0$ given by $\cF \mapsto \cF^{-1}(0)$ is smooth and induces a well defined map
	\begin{align*}
		\cM(\B)_0 \backslash \cM(\B) &\rightarrow \B \\
			\cM(\B)_0 \cF &\mapsto \cF^{-1}(0),
	\end{align*}
	which is a diffeomorphism.
\end{theorem}
\begin{proof}
	Note that Theorem~\ref{thm:cM(B)action}(1) and Lemma~\ref{lem:inverseorbitat0} show that the second assignment in the statement is a well defined bijection. It remains to prove the smoothness claims.
	
	As before, we consider the diffeomorphism 
	$\varphi : \B \times \Spe(1) \rightarrow \cM(\B)$ introduced in Proposition~\ref{prop:cM(B)isBxSp(1)assets} for which $\varphi(a,u)$ has $a$ as unique vanishing point (see the proof of Lemma~\ref{lem:inverseorbitat0}). This property implies that $\varphi$ followed by the map
	\begin{align*}
		\cM(\B) &\rightarrow \B  \\
			\cF &\mapsto \cF^{-1}(0),
	\end{align*}
	is the projection $\B \times \Spe(1) \rightarrow \B$ onto the first factor. Since $\varphi$ is a diffeomorphism, it follows that the previous assignment is indeed smooth. This establishes the first claim of our statement.
		
	We will use the isomorphism of Lie groups $\rho : \Spe(1) \rightarrow \cM(\B)_0$ introduced in Theorem~\ref{thm:cM(B)action}(3). From the definition of $\varphi$ and $\rho$ it is straightforward to see that
	\[
		\varphi(a,uv) = \rho(v)^{-1}\circ\varphi(a,u)
	\]
	for all $a \in \B$ and $u,v \in \Spe(1)$. In other words, $\varphi$ is $\rho$-equivariant for the right $\Spe(1)$-action on $\B \times \Spe(1)$ (defined by acting on the second factor) and the left $\cM(\B)_0$-action on $\cM(\B)$ (defined in Proposition~\ref{prop:cM(B)0actingoncM(B)}). Hence,
	$\varphi$ induces a bijection $\widehat{\varphi} : \B \rightarrow \cM(\B)_0 \backslash \cM(\B)$ such that the diagram
	\begin{equation}\label{eq:BascM(B)quotient}
	\xymatrix{
		\B \times \Spe(1) \ar[d] \ar[r]^-{\varphi} 
			& \cM(\B) \ar[d] \\
			\B \ar[r]^-{\widehat{\varphi}} 
			& \cM(\B)_0 \backslash \cM(\B)
	}
	\end{equation}
	is commutative,	where the vertical arrows are the fiber bundle projections. We note that the left vertical arrow has natural global sections given by $a \mapsto (a,u_0)$, where $u_0 \in \Spe(1)$ is any fixed element. By the definition of a section, we can replace the left vertical arrow with any such global section to still obtain a commutative diagram. This construction shows that $\widehat{\varphi}$ is smooth. Similarly, one can use local sections of the right vertical arrow to prove that the inverse of $\widehat{\varphi}$ is smooth as well, thus concluding that it is in fact a diffeomorphism.
	
	To compute an explicit expression for $\widehat{\varphi}$, let us consider the global section of the left vertical arrow above given by $a \mapsto (a,1)$. Using the latter we conclude, by the diagram commutativity explained above, that $\widehat{\varphi}$ is given by the following sequence of assignments
	\[
		\widehat{\varphi} : a \mapsto (a,1) \mapsto \varphi(a,1) \mapsto \cM(\B)_0 \varphi(a,1).
	\]
	This implies that 
	\[
		\widehat{\varphi}^{-1}(\cM(\B)_0 \varphi(a,1)) = a
			= \varphi(a,1)^{-1}(0),
	\]
	for every $a \in \B$. Since the expression for $\widehat{\varphi}^{-1}$ does not depend on the class representative, we conclude that
	\[
		\widehat{\varphi}^{-1}(\cM(\B)_0 \cF) = \cF^{-1}(0),
	\]
	for every $\cF \in \cM(\B)$. Hence, the assignment $\cM(\B)_0 \backslash \cM(\B) \rightarrow \B$ from our statement is exactly $\widehat{\varphi}^{-1}$ and so it is a diffeomorphism.
\end{proof}

The proof of Theorem~\ref{thm:BascM(B)quotient} implies the next result. 

\begin{corollary}\label{cor:cM(B)asbundleoverB}
	The inverse orbit map $\cM(\B) \rightarrow \B$ given by $\cF \mapsto \cF^{-1}(0)$ yields a globally trivial principal fiber bundle with structure group $\cM(\B)_0 \simeq \Spe(1)$ for its left action on $\cM(\B)$.
\end{corollary}
\begin{proof}
	The commutative diagram \eqref{eq:BascM(B)quotient} has diffeomorphisms as horizontal arrows with the top one equivariant with respect to the corresponding group actions. It is now a simple exercise to prove that the right vertical arrow defines a principal fiber bundle. Furthermore, every section of one of the bundles induces a corresponding one for the other: this is easily achieved using $\varphi$ and $\widehat{\varphi}$. As noted in the proof of Theorem~\ref{thm:BascM(B)quotient}, the principal fiber bundle in the left of \eqref{eq:BascM(B)quotient} admits a global section and so the same holds for the principal fiber bundle in the right.
\end{proof}

\begin{remark}\label{rmk:cM(B)asbundleoverB}
	It is useful to compare Corollary~\ref{cor:cM(B)asbundleoverB} with Proposition~\ref{prop:BasM(B)quotient}. From the latter, applying the final observations of subsection~\ref{subsec:principalbundles}, it follows that the inverse orbit map
	\begin{align}
		\Mbb(\B) &\rightarrow \B  \label{eq:Mbb(B)asbundleoverB}  \\
			F &\mapsto F^{-1}(0), \notag
	\end{align}
	yields a principal fiber bundle with structure group $\Mbb(\B)_0$ acting on the left. On the other hand, \cite[Theorem~1.1,~Chapter~VI]{Helgason} implies the global triviality of the principal fiber bundle \eqref{eq:Mbb(B)asbundleoverB}. The latter is in fact a general property associated to certain Riemannian symmetric spaces; we refer to \cite[Chapter~VI]{Helgason} for further details. Hence, the behavior exhibited in Corollary~\ref{cor:cM(B)asbundleoverB} for $\B$, in the setup of (slice) regular M\"obius transformations, is similar to the classical setup of Riemannian symmetric spaces.
\end{remark}

Next, we show that slice regular M\"obius transformations allow to write $\B$ as a smooth quotient of the Lie group $\Spe(1,1)$. This will require to use double cosets instead of the usual one-sided cosets. We will consider $\Spe(1) \times \{1\}$ as a diagonally embedded subgroup of $\Spe(1,1)$. For this, we will denote by $\operatorname{diag}(a,b)$ the $2 \times 2$ diagonal matrix with diagonal entries $(a,b)$. 

\begin{theorem}\label{thm:BasSp(1,1)quotient}
	The double coset space $(\Spe(1) \times \{1\}) \backslash \Spe(1,1) / \Spe(1) I_2$ has a natural manifold structure obtained from Proposition~\ref{prop:buildprincipalbundle}. Moreover, the map given~by
	\begin{align*}
		(\Spe(1) \times \{1\}) \backslash \Spe(1,1) / \Spe(1) I_2 
			&\rightarrow \B  \\
		(\Spe(1) \times \{1\}) A \Spe(1) I_2
			&\mapsto \big(\cF_{A^{-1}}\big)^{-1}(0),
	\end{align*}
	is a well defined diffeomorphism.
\end{theorem}
\begin{proof}
	Let us denote by $\psi : \Spe(1,1)/ \Spe(1) I_2 \rightarrow \cM(\B)$ the map considered in Theorem~\ref{thm:cM(B)manifold} that yields the manifold structure of $\cM(\B)$. Hence, $\psi$ is the diffeomorphism given by
	\[
		\psi(A\Spe(1) I_2) = \cF_{A^{-1}},
	\]
	for every $A \in \Spe(1,1)$. Let us also consider the map
	\begin{align*}
		\widetilde{\rho} : \Spe(1) \times \{1\}
				&\rightarrow \cM(\B)_0  \\
			\widetilde{\rho}(u,1) & = R_{u^{-1}}.
	\end{align*}
	Hence, Theorem~\ref{thm:cM(B)action}(3) implies that $\widetilde{\rho}$ is an isomorphism of Lie groups. 
	
	With this notation, we can compute as follows
	\begin{align*}
		\psi\big(\operatorname{diag}(u,1)A \Spe(1) I_2\big)
			&= \cF_{\big(\operatorname{diag}(u,1)A\big)^{-1}}  \\
			&= R_{u^{-1}} \circ \cF_{A^{-1}}  \\
			&= \widetilde{\rho}(u,1) \circ \psi\big(A\Spe(1) I_2\big),
	\end{align*}
	for every $u \in \Spe(1)$ and $A \in \Spe(1,1)$. In other words, $\psi$ is a $\widetilde{\rho}$-equivariant diffeomorphism. By Proposition~\ref{prop:cM(B)0actingoncM(B)} the action of $\cM(\B)_0 = \widetilde{\rho}(\Spe(1) \times \{1\})$ on $\cM(\B)$ is proper and free. Hence, we conclude that the $\Spe(1) \times \{1\}$-action on $\Spe(1,1)/\Spe(1) I_2$ is proper and free as well. This yields a natural manifold structure on the double coset space $(\Spe(1) \times \{1\}) \backslash \big(\Spe(1,1) / \Spe(1) I_2\big)$ once we apply Proposition~\ref{prop:buildprincipalbundle}. Similarly, one can show that $(\Spe(1) \times \{1\}) \backslash \Spe(1,1)$ is a manifold with a proper and free $\Spe(1) I_2$-action that thus yields, by Proposition~\ref{prop:buildprincipalbundle}, a natural manifold structure on $\big((\Spe(1) \times \{1\}) \backslash \Spe(1,1)\big) / \Spe(1) I_2$. It is a standard exercise to prove that there is a canonical diffeomorphism
	\[
		(\Spe(1) \times \{1\}) \backslash \big(\Spe(1,1) / \Spe(1) I_2\big) \simeq
		\big((\Spe(1) \times \{1\}) \backslash \Spe(1,1)\big) / \Spe(1) I_2,
	\]
	between the two double coset spaces just introduced. This yields a (unique) natural manifold structure on 
	\[
		(\Spe(1) \times \{1\}) \backslash \Spe(1,1) / \Spe(1) I_2,
	\]	
 	where we have dropped the (two possible sets of) parentheses because of the diffeomorphism above. This establishes the first claim of the statement.
 	
 	An argument similar to the one used in the proof of Theorem~\ref{thm:BascM(B)quotient} and its diagram \eqref{eq:BascM(B)quotient} allows to conclude that $\psi$ induces the diffeomorphism given by
	\begin{align*}
		\widehat{\psi} : (\Spe(1) \times \{1\}) \backslash 
				\Spe(1,1)/\Spe(1) I_2 
					&\rightarrow \cM(\B)_0\backslash\cM(\B)  \\
			(\Spe(1) \times \{1\}) A \Spe(1) I_2 
					&\mapsto \cM(\B)_0 \cF_{A^{-1}}.
	\end{align*}
	We now observe that the map from the statement is the composition of $\widehat{\psi}$ with the diffeomorphism from Theorem~\ref{thm:BascM(B)quotient}, and so it is a diffeomorphism as well. This proves the second claim and completes the proof of the statement.
\end{proof}

\begin{remark}\label{rmk:BasSp(1,1)quotient-sliceregular}
	It is important to highlight the similarities and differences in the realization of $\B$ as a quotient of other groups and manifolds using either the classical or the slice regular M\"obius transformations.
	
	In the first place, we observe that Proposition~\ref{prop:BasM(B)quotient} and Theorem~\ref{thm:BascM(B)quotient} give the same sort of result: in both cases, $\B$ is the quotient obtained by modding out using a compact Lie group. However, in the classical case $\Mbb(\B)$ is also itself a group and in the slice regular case $\cM(\B)$ is just a manifold. On the other hand, for the classical case we consider a quotient of Lie groups, but in the slice regular case we use the quotient associated to a principal fiber bundle.
	
	Next, we can compare Proposition~\ref{prop:BasSp(1,1)quotient} and Theorem~\ref{thm:BasSp(1,1)quotient}. In both cases, we express $\B$ as a quotient of the Lie group $\Spe(1,1)$. As noted in Remark~\ref{rmk:BasSp(1,1)quotient}, the quotient in the classical case can be naturally reduced to a one-sided coset given by~\eqref{eq:BasSp(1,1)quotient}: this uses the fact that $\Z_2 I_2$ is central. However, for the slice regular case the double coset cannot be further simplified. As noted in Remark~\ref{rmk:cM(B)manifold} this follows from the non-commutativity of $\HH$, which implies that $\Spe(1) I_2$ is not a central subgroup of $\Spe(1,1)$. 
	
	Hence, we have obtained two similar, but at the same time quite different, expressions of $\B$ as quotients of $\Spe(1,1)$, and these are
	\begin{align}
		\B &\simeq (\Spe(1) \times \Spe(1)) \backslash \Spe(1,1), 
				&\text{ classical case,}  \label{eq:Bclassicalquotient}  \\
		\B &\simeq (\Spe(1) \times \{1\})
				\backslash \Spe(1,1) /\Spe(1) I_2 
				&\text{ slice regular case.}  
						\label{eq:Bsliceregularquotient}
	\end{align}
	We observe that
	\[
		(\Spe(1) \times \{1\}) \Spe(1) I_2 = \Spe(1) \times \Spe(1),
	\]
	but the lack of commutativity of $\HH$ prohibits us from using this identity to relate \eqref{eq:Bclassicalquotient} and \eqref{eq:Bsliceregularquotient}. Hence, from the point of view of Lie theory, these quotient realizations of $\B$ differ from each other. In particular, the same will occur for any structure derived from \eqref{eq:Bclassicalquotient} and \eqref{eq:Bsliceregularquotient}. It is interesting to note that both of these expressions take a quotient using two copies of $\Spe(1)$: the former takes quotient on the left side only and the latter uses one copy of $\Spe(1)$ on each side to take the quotient. Again, that these constructions are not the same comes from the non-commutativity of $\HH$.
\end{remark}

\begin{remark}\label{rmk:FvsF-1}
	The quotient realizations of $\cM(\B)$ and $\B$, obtained in Theorems~\ref{thm:cM(B)manifold} and \ref{thm:BascM(B)quotient} have used the assignments
	\[
	A\Spe(1) I_2 \mapsto \cF_{A^{-1}}, \quad
	\cM(\B)_0\cF \mapsto \cF^{-1}(0),
	\]
	respectively, which can be considered as non-standard because of the introduction of an inverse. We will now explain the reason for our choice.
	
	For the classical case, we expressed $\B$ as a quotient of a Lie group in Proposition~\ref{prop:BasM(B)quotient} using the non-standard inverse orbit map at $0$, which is given by
	\begin{align*}
		\Mbb(\B) &\rightarrow \B  \\
			F &\mapsto F^{-1}(0).
	\end{align*}
	However, it is most frequently considered the usual orbit map given~by
	\begin{align*}
		\Mbb(\B) &\rightarrow \B  \\
			F &\mapsto F(0).
	\end{align*}
	To achieve the bijection of $\B$ with a quotient of $\Mbb(\B)$ one considers, in the former case, the sequence of equivalences
	\begin{align*}
		F_1^{-1}(0) = F_2^{-1}(0)
			&\iff F_2 \circ F_1^{-1}(0) = 0  \\
			&\iff F_2 \circ F_1^{-1} \in \Mbb(\B)_0  \\
			&\iff \Mbb(\B)_0 F_2 = \Mbb(\B)_0 F_1,
	\end{align*}
	and for the latter case one considers
	\begin{align*}
		F_1(0) = F_2(0)
			&\iff F_2^{-1} \circ F_1(0) = 0  \\
			&\iff F_2^{-1} \circ F_1 \in \Mbb(\B)_0  \\
			&\iff F_2 \Mbb(\B)_0 = F_1 \Mbb(\B)_0,
	\end{align*}
	which yield diffeomorphisms between $\B$ and the quotient spaces $\Mbb(\B)_0\backslash \Mbb(\B)$ (as obtained in Proposition~\ref{prop:BasM(B)quotient}) and $\Mbb(\B) / \Mbb(\B)_0$, respectively. Both arguments are possible because $\Mbb(\B)$ is a group, so that both compositions $F_2 \circ F_1^{-1}$ and $F_2^{-1} \circ F_1$ belong to $\Mbb(\B)$ and we can further consider whether either of them belongs to the subgroup $\Mbb(\B)_0$.
	
	For the slice regular case, the manifold $\cM(\B)$ is not a group under the composition of maps (see~Remark~\ref{rmk:regMobius-composition}), and so the argument used above for $\Mbb(\B)$ cannot go through. This difficulty is overcame by the use of Lemma~\ref{lem:inverseorbitat0} that yields the equivalence
	\begin{align*}
		\cF_1^{-1}(0) = \cF_2^{-1}(0) 
				&\iff \exists\; \cF_0 \in \cM(\B)_0 
							\text{ such that }
							\cF_1 = \cF_0 \circ \cF_2  \\  
				&\iff \cM(\B)_0 \cF_1 = \cM(\B)_0 \cF_2,
	\end{align*}
	which is key to obtain Theorem~\ref{thm:BascM(B)quotient}. On the other hand, compositions of the form $\cF_2 \circ \cF_0$, where $\cF_2 \in \cM(\B)$ and $\cF_0 \in \cM(\B)_0$, in general do not belong to $\cM(\B)$. The latter was noted in Remark~\ref{rmk:regMobius-composition}. Hence, we necessarily have to use the inverse orbit map at $0$ considered in Theorem~\ref{thm:BascM(B)quotient}.
	
	Note that we have introduced an inverse in Theorem~\ref{thm:cM(B)manifold} to express $\cM(\B)$ as a quotient of $\Spe(1,1)$. In this case, it is easy to verify that we do have the choice to omit such inverse and still obtain a similar quotient. Our reason to consider the inverse in this case is twofold. In the first place, the inverses considered in both Theorems~\ref{thm:cM(B)manifold} and \ref{thm:BascM(B)quotient} can be thought as offsetting each other when expressing $\B$ as quotient of $\Spe(1,1)$ in Theorem~\ref{thm:BasSp(1,1)quotient}, even though they cannot actually cancel each other. Secondly, if we had omitted the inverse in Theorem~\ref{thm:cM(B)manifold}, then Theorem~\ref{thm:BasSp(1,1)quotient} would have yielded a diffeomorphism of $\B$ with a double for which both subgroups act on the left. All of this considered, we believe our choice of inverse maps provide the simplest expression of $\B$ as a quotient of the Lie group $\Spe(1,1)$.
\end{remark}

\subsection*{Acknowledgment}
This research was partially supported by SNI-Conahcyt and by Conahcyt Grants 280732 and 61517.

\end{document}